\documentclass[12pt]{article}
\usepackage{amsmath}
\usepackage{amssymb}
\usepackage{amsthm}

\newtheorem{thm}{Theorem}
\newtheorem{prop}[thm]{Proposition}

\newtheorem{dfn}[thm]{Definition}
\newtheorem{lemm}[thm]{Lemma}

\def\be{\begin{equation}}
\def\ee{\end{equation}}

\def\R{\mathbb{R}}
\def\Ri{\mathbb{R}\cup \{+\infty\}}

\def\pa{\partial }
\def\dom{\mathrm{dom}\,}

\def\l{\langle}
\def\r{\rangle}

\def\cl{\mathrm{cl}\,}

\def\d{\delta}
\def\la{\lambda}
\def\ol{\overline}

\begin{document}
\title{Barrier functions in the subdifferential theory}

\author{Milen Ivanov\thanks{Supported by Bulgarian National Scientific Fund under grant KP-06-Í22/4.}\\
{\footnotesize Radiant Life Technologies Ltd.}\\
{\footnotesize Nicosia, Cyprus}\\
{\footnotesize milen@radiant-life-technologies.com}\\
      \and
        Nadia Zlateva\thanks{Supported by  Scientific Fund of Sofia University under grant for 2019.}\\
{\footnotesize Faculty of Mathematics and Informatics}\\
{\footnotesize Sofia University}\\
{\footnotesize 5, James Bourchier Blvd.}\\
{\footnotesize 1164 Sofia, Bulgaria}\\
{\footnotesize zlateva@fmi.uni-sofia.bg}\\
        }
\date{}

\maketitle

\begin{abstract}
	We present a new method for proving Correa-Jofr\'e-Thibault theorem that monotonicity of subdifferential implies convexity of the function.
	
	This new method is based on barrier functions. Barrier functions help overcome some of the main technical difficulties when working with lower semicontinuous functions.
\end{abstract}

\section{Introduction}\label{intro}

In 1990's Correa, Jofr\'e and Thibault in series of papers proved that a convex lower semicontinuous function can be characterized by monotonicity property of its  subdifferential -- in reflexive Banach space for the Clarke subdifferential in \cite{C-J-T-1} and in  any Banach space for axiomatically introduced  subdifferential  in \cite{C-J-T-2} and for more general  axiomatic presubdifferential in \cite{C-J-T-3}.

The main tool for proving this characterization is the  Mean Value Theorem of Zagrodny~\cite{zagrodni} which holds for a lower semicontinuous function in a Banach space for any presubdifferential, see \cite{thibault-MVT}.

Jules and Lassonde prove a subdifferential test for optimality, see \cite{ju_lasond} of Minty type \cite{minty} involving a subdifferential satisfying certain axioms.

In setting up the axiomatic framework we follow \cite{tz}, but  we pick the apparently minimal set of axioms under which proofs can work. In this way our results are slightly more general. As mentioned below, adding another natural axiom can significantly simplify the presentation.

%Let $f:X\to\R\cup\left\{ +\infty\right\} $ be a proper, convex and lower semicontinuous function.

We work in a real Banach space $X$ with dual $X^*$.

\begin{dfn}[axioms for subdifferential]\label{dfn:1}
Multi-valued operator $\pa $  which associates to any function $f:X\to \Ri$ and any $x\in X$ a (possibly empty) subset $\pa f(x)\subset X^*$ is feasible subdifferential if
\begin{enumerate}
  \item [\rm (P1)] $\pa f(x)=\pa ^c f(x)$  whenever $f$ is a convex and continuous function in a neighbourhood $U$ of $x$, where $\pa ^c$ stands for the Fenchel subdifferential, i.e.
 \[
 \pa ^c f(x):=\{ x^*\in X^*:\l x^*, y-x\r \le f(y)-f(x), \ \forall y\in U\};
 \]
 %      \item [\rm P2.] If $f=g+const$ in a neighbourhood of $ x$, then $\pa f(x)=\pa g(x)$;
  \item [\rm (P2)] For $f$ lower semicontinuous and $g$  convex and continuous in a neighbourhood of  $x\in\dom f$,
  \begin{equation}\label{eq:feas}
    0\in \pa g(x)+\limsup_{y\to x,\ f(y)\to f(x)} \pa f(y)
  \end{equation}
  whenever $x$ is a local minimum point of $f+g$.
\end{enumerate}
\end{dfn}

In more details \eqref{eq:feas} means that there are $y_n^*\in\pa f(y_n)$ such that $y_n\to x$, $f(y_n)\to f(x)$ and the sequence $y_n^*$ converges in the $w^*$ topology to some $y^*$ such that $-y^*\in \pa g(x)$.

Note  that all presubdifferentials considered by Correa, Jofr\'e and Thibault in \cite{C-J-T-1,C-J-T-2,C-J-T-3}, as well as subdifferentials considered by Jules and Lassonde in \cite{ju_lasond} are feasible subdifferentials in the  sense of Definition~\ref{dfn:1}.

In terms of Ioffe's extensive classification, see \cite{ioffe-subd-class}, (P1) is called \textit{contiguity}, while (P2) is weak-star form of \textit{trustworthness}.

We prove the above mentioned results for feasible subdifferentials and in a different and unified way -- by using barrier functions instead of (a variant of) Zagrodny Theorem.

For convenience of notation  we often identify the map $\pa f:X\rightrightarrows X^*$ with its graph, that is, $(x,x^*)\in \pa f$ is a shorthand for $x^*\in \pa f(x)$.

\medskip
The main contribution of this work is a new method, based on \cite{iz:max-mon}, see also \cite[p.569]{penot:book}, for proving the following result of Correa, Jofr\'e and Thibault.

\begin{thm}[Correa-Jofr\'e-Thibault]\label{thm:CJT}
	Let $X$ be a Banach space and let $\pa $ be a feasible subdifferential.
	
	Let $f:X\to\R\cup\left\{ +\infty\right\} $ be a proper lower semicontinuous function.
	
	If $\pa f$ is monotone, then $f$ is convex.
\end{thm}

Recall that monotonicity of $\pa f$ means
$$
	(x_i,x_i^*)\in \pa f,\ i=1,2 \Rightarrow \langle x_2^*-x_1^*,x_2-x_1\rangle\ge 0.
$$

The routine way of demonstrating the above result can be sketched like this: examining the proof in \cite{thibault-MVT} it is clear that Zagrodny Mean Value Theorem holds for any feasible subdifferential. Using it, one can prove a Minty test for optimality \cite{minty} of the following form
$$
	\langle x^\ast, x-x_0\rangle \ge 0,\ \forall (x,x^*)\in\pa f \Rightarrow 0\in \pa^c f(x_0),
$$
where $f$ is proper and lower semicontinuous and $\pa$ is feasible subdifferential.

If -- on top of (P1) and (P2) -- $\pa$ also satisfies the natural axiom
\begin{enumerate}
\item[\rm (P3)]
    $\pa(f+x^*) = \pa f + x^*,\quad\forall x^*\in X^*$
\end{enumerate}
(called in \cite{ju_lasond} \emph{stability property} which is rather limited form of \textit{calculability} axiom in \cite{ioffe-subd-class}), then from Minty test immediately follows that $\partial\cup \partial ^c$ has the somewhat surprising property (first noted by Jules and Lasonde) to be \textit{maximal with respect to monotonicity relation}. That is, if $(x_0,x_0^*)$ is monotonous related to $\pa f$:
$$
	\langle x^*-x_0^*,x-x_0\rangle \ge 0, \quad \forall (x,x^*)\in \pa f,
$$
then $(x_0,x_0^*)\in \pa^c f$.  From the latter Moreau-Rockafellar Theorem about maximal monotonicity of $\pa^cf$ for convex, proper and lower semicontinuous $f$ follows immediately, but the surprising fact is that the above is true even if $\pa f$ is not itself monotone. (For precise statement see Theorem~\ref{thm:jl:mon}.)

So, in particular if $\pa$ is feasible and satisfies in addition  (P3), then
$$\pa f\subset \pa ^c f,$$
whenever  $\pa f$ is monotone. Further the proof can  be completed as we do in here presented  proof of Theorem~\ref{thm:CJT}.

Note that the additional axiom (P3) is not really necessary.

Our approach is based on a different technique involving barrier functions instead of Zagrodny Theorem. We will also make the effort to obtain Theorem~\ref{thm:CJT} for general feasible subdifferental.

The paper is organized as follows.

In Section~\ref{sec:barier} we construct and consider the class of barrier functions we use. In Section~\ref{sec:special} we show some additional properties of the feasible subdifferential linked to (P2) axiom. Finally, in Section~\ref{sec:main} we present our proof of Correa-Jofr\'e-Thibault theorem.

\section{Barrier functions}\label{sec:barier}

Let $U\subset X$ be a open, convex  and bounded neighbourhood of $0$, i.e. $0\in U$. Let $\mu$ be the \textit{Minkowski functional} of $U$, that is,
\[
\mu (x)=\inf \{ \la : \la >0,\ x\in \la U\} \mbox{ for }x\neq0,\ \mu(0)=0.
\]
It is clear that $\mu (x) < 1$ if $x\in U$,  and $\mu (x)\ge 1$ if $x\not \in U$.

Let us first list few properties of Minkowski functional, see e.g. \cite{Clarke}:
\begin{itemize}
  \item [(i)] $\mu $ has values in $[0,+\infty)$;
  \item [(ii)] $\mu $ is positively homogeneous, i.e. $\mu (tx)=t\mu (x)$ for all $x\in X$ and $t\ge 0$;
  \item[(iii)] $\mu(x+y)\le \mu(x)+\mu (y)$ for all $x,y\in X$ and hence $\mu $ is convex;
     \item[(iv)] $\{ x: \mu (x)\le 1\}=\cl U\Rightarrow \{ x: \mu (x) = 1\}=\pa U:=\cl U\setminus U$, where $\cl U$ denotes the topological closure of $U$ and $\pa U$ denotes its boundary.
   \end{itemize}

   Moreover,
   \begin{itemize}
  \item [(v)] There exists $c>0$ such that $\mu (x)\le c\| x\|$ for all $x\in X$;
  \item [(vi)] There exists $b>0$ such that $\mu (x)\ge b\| x\|$ for all $x\in X$;
   \item [(vii)] $\mu$ is  Lipschtz  continuous.
\end{itemize}

Indeed, let $\d >0$ be such that $\d B_X\subset U$, where $B_X:=\{ x\in X:\| x\| \le 1\}$. Then $\displaystyle \frac{\delta x}{\| x\|}\in \d B_X\subset U$, so $\displaystyle \mu \left( \frac{\delta x}{\| x\|} \right) \le 1$, and $\mu (x)\le \delta^{-1}\| x\|$, which is (v) with $c:=\delta^{-1}$.

Further, let $U\subset sB_X$, then $\displaystyle \frac{sx}{\| x\|} \not \in U$, so $\displaystyle \mu \left( \frac{sx}{\| x\| }\right) \ge 1$, and $\displaystyle \mu (x)\ge \frac{1}{s}\| x\|$ giving (vi) with $b:=s^{-1}$.

Finally, $\mu (x)=\mu ((x-y)+y)\le \mu (x-y)+\mu (y)$ by (ii). Hence, $\mu (x)-\mu(y)\le \mu (x-y)$ and using (vi) we get $\mu (x)-\mu(y)\le b\| x-y\|$ which yields $|\mu (x)-\mu(y)|\le b\| x-y\|$ for all $x,y\in X$ and (vii) holds.

\medskip

For $x\in U$,  define the function
\begin{equation}\label{def:k}
  k(x):=\frac{\mu (x)}{1-\mu (x)}=\frac{1}{1-\mu (x)}-1.
\end{equation}

The next lemma shows that $k$ is a \textit{barrier function} for $U$, i.e. a continuous function whose values tend to infinity while the arguments tend to $\pa U$ (see e.g. \cite{nocri}).

\begin{lemm}\label{lem:k_proberties}
The function $k$ defined by \eqref{def:k} has the following properties:
\begin{itemize}
  \item[{\rm (a)}] $k(0)=0$ and for some $b>0$, $k(x)\ge \mu(x)\ge b\| x\|$.
  \item[{\rm (b)}] $\displaystyle \lim _{x\to \pa U} k(x)=+\infty$;
   \item[{\rm (c)}] $k$ is convex and Lipschitz on each level set $\{ x: k(x)\le \lambda\}$;
  \item[{\rm (d)}] the function $\displaystyle \ol k(x):=\left\{ \begin{array}{ll} k(x),&x\in  U\\ +\infty,&x\in X\setminus U\end{array}\right.$ is lower semicontinuous and convex.

\end{itemize}
\end{lemm}
\begin{proof}
The properties (a) and (b) are clear.

In order to prove that $k$ is convex, it is enough to show that the function $\displaystyle g(x):=k(x)+1= \frac{1}{1-\mu (x)}$ is convex.

To this end, fix $x,\ y\in U$ and $\la \in [0,1]$.
From~(i)
\[
\mu (\la x+(1-\la )y)\le \la \mu (x)+(1-\la )\mu (y) \Leftrightarrow
\]
\[
1-\mu (\la x+(1-\la )y)\ge \la (1-\mu (x))+(1-\la )(1-\mu (y) ).
\]
Set $\alpha:=1-\mu(x)$, $\beta:=1-\mu (y)$, so $\alpha,\beta\in (0,1)$. Then the latter says
\be\label{ge}
g(\la x+(1-\la) y)\le \frac{1}{\la \alpha+(1-\la)\beta}.
\ee

We claim that
\be\label{star}
\frac{1}{\la \alpha+(1-\la) \beta}\le \frac{\la}{\alpha}+\frac{1-\la}{\beta}.
\ee
Of course, (\ref{star}) is equivalent to
\begin{eqnarray*}
0&\le& \la \beta(\la \alpha+(1-\la)\beta)+(1-\la)\alpha(\la \alpha+(1-\la )\beta)-\alpha\beta\\
 & = &\la (1-\la)(\alpha^2+\beta^2)+(\la^2+(1-\la)^2-1)\alpha\beta\\
 &=&\la (1-\la)(\alpha^2+\beta^2)+(2\la^2-2\la)\alpha\beta\\
 &=&\la(1-\la)(\alpha-\beta)^2.
\end{eqnarray*}
From (\ref{ge}) and (\ref{star}) it follows that
\[
g(\la x+(1-\la )y)\le \la g(x)+(1-\la )g(y)
\]
and $k$ is convex. Lipschitz continuity of $k$ on each level set is inherited by the Lipschitz continuity of $\mu$.

(d) follows from (b) and (c).
\end{proof}

\section{Properties of feasible subdifferential}\label{sec:special}

For the sake of clarity, we will take two technical parts out of the proof of the main result. Namely, we will show some additional properties of feasible subdifferential linked to (P2) axiom.

\begin{lemm}\label{lem:p2-reform}
	Let $\pa$ be a feasible subdifferential. Let $f:X\to \R\cup\{+\infty \}$ be lower semicontinuous and bounded below on the open, convex and bounded set $U\subset X$. Let, moreover, $\dom f \cap U\neq\emptyset$. Let $g$ be convex continuous and bounded below barrier function for $U$. Let $\bar x\in X$be fixed.
	
	Then there exist sequences $(x_n)_{n=1}^\infty,(y_n)_{n=1}^\infty\subset U$ and $x_n^*\in\pa f(x_n)$, $y_n^*\in\pa ^c g(y_n)$ such that
	\begin{equation}
		\label{eq:t-l:xy:close}
		\|x_n-y_n\|\to 0,
	\end{equation}
	\begin{equation}
		\label{eq:t-l:inf}
		f(x_n)+g(y_n)\to \inf_U(f+g),
	\end{equation}
		\begin{equation}
			\label{eq:t-l:sum}
			\langle x_n^*,x_n-\bar x\rangle + \langle y_n^*,y_n-\bar x\rangle\to 0.
		\end{equation}
\end{lemm}
\begin{proof}
	Define
	\begin{equation}
		\label{eq:bigM-def}
		M:= \sup \{\|y-\bar x\| :\ y\in U\}.
	\end{equation}
	Since $U$ is bounded, $M\in\R$.
	
	Fix a sequence $(\varepsilon_n)_{n=1}^\infty$ such that $\varepsilon_n>0$ and $\varepsilon_n\to 0$ as $n\to\infty$.
	
	Pick $z_n\in U$ such that
	$$
		f(z_n) + g(z_n) < \inf_U (f+g) + \varepsilon_n.
	$$
	
	By Ekeland Variational Principle there are $y_n\in U$ such that
	\begin{equation}
		\label{eq:eps-min}
		f(y_n) + g(y_n) < \inf_U (f+g) + \varepsilon_n,
	\end{equation}
	and the function
	$$
		x \to f(x) + g(x) + \varepsilon\|x-y_n\|
	$$
	attains its minimum at $y_n$. By (P2) there are $u_k\to y_n$, as $k\to \infty$, and $u_k^*\in\pa f(u_k)$ such that $f(u_k)\to f(y_n)$ and
	$$
		w^*{-}\lim u_k^*=u^*\in -\pa(g+\varepsilon_n\|\cdot-y_n\|)(y_n).
	$$
	By Sum Theorem for the Fenchel subdifferential and the fact that $\pa\varepsilon_n\|\cdot-y_n\|=\varepsilon_n\pa \|\cdot-y_n\| $ it follows that there are $y_n^*\in\pa g(y_n)$ such that
	\begin{equation}
		\label{eq:xn-close-u}
		\|u^*+y_n^*\|\le\varepsilon_n.
	\end{equation}
	Since $(u_k,f(u_k))\to (y_n,f(y_n))$ as $k\to\infty$, there is $k_1\in\mathbb{N}$ such that
	\begin{equation}
		\label{eq:u-conv}
		\|u_k-y_n\|<\varepsilon_n,\ |f(u_k)-f(y_n)|<\varepsilon_n,\quad\forall k>k_1.
	\end{equation}
	By \eqref{eq:eps-min} we have
	\begin{equation}
		\label{eq:eps-min-y}
		|f(u_k) + g(y_n) - \inf_U (f+g)| \le 2 \varepsilon_n, \quad\forall k>k_1.
	\end{equation}
	Note that
	$$
		\langle u_k^*,u_k-\bar x\rangle + \langle y_n^*,y_n-\bar x\rangle = \langle u^*+y_n^*,y_n-\bar x\rangle + \langle u_k^*-u^*,y_n-\bar x\rangle + \langle u_k^*,u_k-y_n\rangle.
	$$
	Now, $|\langle u^*+y_n^*,y_n-\bar x\rangle|\le\|u^*+y_n^*\|\|y_n-\bar x\|\le M\varepsilon_n$ by \eqref{eq:bigM-def} and \eqref{eq:xn-close-u}. Since $u_k^*$ weak-star converges to $u^*$, we have $\langle u_k^*-u^*,y_n-\bar x\rangle\to 0$ as $k\to\infty$. Also, since by Banach-Steinhaus  Theorem the sequence $(u_k^*)_{k=1}^\infty$ is bounded, and $u_k\to y_n$ as $k\to\infty$, we have $|\langle u_k^*,u_k-y_n\rangle|\le \|u_k^*\|\|u_k-y_n\|\to 0$  as $k\to\infty$. Therefore, there is $k_2\in\mathbb{N}$ such that
	\begin{equation}
		\label{eq:tech-last}
		|\langle u_k^*,u_k-\bar x\rangle + \langle y_n^*,y_n-\bar x\rangle| \le (M+1) \varepsilon_n, \quad\forall k>k_2.
	\end{equation}
	Set
	$$
		\bar k = \max \{k_1,k_2\}+1\mbox{ and } (x_n,x_n^*) := (u_{\bar k},u_{\bar k}^*).
	$$
	From \eqref{eq:u-conv}, \eqref{eq:eps-min-y} and \eqref{eq:tech-last} it follows that so constructed sequences $(x_n,x_n^*)_{n=1}^\infty$ and $(y_n,y_n^*)_{n=1}^\infty$ satisfy \eqref{eq:t-l:xy:close}, \eqref{eq:t-l:inf} and \eqref{eq:t-l:sum}.	
\end{proof}

\begin{prop}
	\label{prop:convex}
	Let $\pa$ be a feasible subdifferential. Let $f:X\to \R\cup\{+\infty \}$ be a proper and lower semicontinuous function. Then $\dom \pa f$ is nonempty and
	\begin{equation}
		\label{eq:conv:form}
		f(\ol x) \le \sup\{f(x) + \langle x^*,\ol x-x\rangle:\ (x,x^*)\in\pa f \}, \mbox{ for all } \ol x\in X.
	\end{equation}	
\end{prop}
\begin{proof}
	Fix arbitrary $\bar x\in X$ and
	$$
		r < f(\bar x).
	$$
	Fix some $y\in\dom f$.
	
	Since $f$ is lower semicontinuous and the segment $[\bar x,y]$ is compact, there is $\delta >0$ such that $f$ is bounded below on the set
	$$
	C := [\bar x,y] + \delta B_X^\circ.
	$$
	Let $k$ be the barrier function defined by \eqref{def:k} for the set $U:=C-\{\bar x\}$.
	
	Let $\xi >0$ be such that $f>r$ on $\bar x + \xi B_X$. Fix $a>0$ and such that
	$$
		a > (r-\inf_U f)/(\xi b),
	$$
where $b>0$ and $k(x)\ge \mu(x)\ge b\| x\|$ (see Lemma~\ref{lem:k_proberties} (a)).

	Then it is immediate that
	\begin{equation}
		\label{eq:fpluskr}
		f(x) + ak(x-\bar x) > r,\quad\forall x\in U.
	\end{equation}
	Apply Lemma~\ref{lem:p2-reform} to $f$ and the barrier function $g=ak(\cdot -\ol x)$ to get $x_n^*\in \pa f(x_n)$ and $y_n^*\in \pa ^c k(\cdot-\bar x)(y_n)$ such that \eqref{eq:t-l:xy:close} and \eqref{eq:t-l:inf} are fulfilled and, moreover,
	\begin{equation}
		\label{eq:prop:1}
		\langle x_n^*,x_n-\bar x\rangle +a \langle y_n^*,y_n-\bar x\rangle= \alpha_n,\mbox{ where }\lim_{n\to\infty}\alpha_n = 0.
	\end{equation}
	Since $y_n^*\in  \pa ^c k (y_n-\bar x)$, we have
	$$
		k(0) \ge k(y_n-\bar x) + \langle y_n^*, \bar x-y_n\rangle \iff \langle y_n^*, y_n-\bar x\rangle  \ge  k(y_n-\bar x),
	$$
	
	From the boundedness below of $f$ and \eqref{eq:t-l:inf} it follows that the sequence $(k(y_n-\bar x))_{n=1}^\infty$ is bounded. Since $k$ is Lipschitz on its level sets (see Lemma~\ref{lem:k_proberties} (c)), from \eqref{eq:t-l:xy:close} it follows that
	$$
		k(y_n-\bar x) = k(x_n-\bar x) + \beta_n,\mbox{ where }\lim_{n\to\infty}\beta_n = 0.
	$$
	These and \eqref{eq:prop:1} give $\langle x_n^*,x_n-\bar x\rangle = \alpha_n - a\langle y_n^*,y_n-\bar x\rangle \le \alpha_n - ak(y_n-\bar x) = - ak(x_n-\bar x) +  \alpha_n -a\beta _n$. So,
	$$
		\langle x_n^*,x_n-\bar x\rangle \le  - ak(x_n-\bar x) +\gamma_n,\mbox{ where }\lim_{n\to\infty}\gamma_n = 0.
	$$
	From \eqref{eq:fpluskr} it follows that $ 	\langle x_n^*,x_n-\bar x\rangle < f(x_n) -r + \gamma_n$, or, equivalently,
	$$
		r <f(x_n) + \langle x_n^*,\bar x- x_n\rangle + \gamma_n.
	$$
	Therefore,
	\begin{eqnarray*}
		r &\le& \sup\{f(x_n) + \langle x_n^*,\bar x- x_n\rangle:\ n\in\mathbb{N}\}\\
		&\le&  \sup\{f(x) + \langle x^*,\bar x-x\rangle:\ (x,x^*)\in\pa f \}.	
	\end{eqnarray*}
	Since $	r < f(\bar x)$ were arbitrary, we are done.
\end{proof}

\section{Monotonicity and convexity}\label{sec:main}

We start with the following extension to the case of feasible subdifferential of a result of Jules and Lassonde \cite{ju_lasond}.

\begin{thm}
	\label{thm:jl:mon}
  Let $\pa$ be a feasible subdifferential. Let $f: X\to \mathbb{R}\cup\{+\infty\}$ be a proper lower semicontinuous function. Let $(x_0,x_0^*)\in X\times X^*$ be in monotone relation to $\pa f$, that is,
  \begin{equation}\label{eq:mon-rel}
    \langle x^*-x_0^*,x-x_0\rangle\ge 0,\quad \forall (x,x^*)\in \pa f.
  \end{equation}
  Then $(x_0,x_0^*)\in\pa^c f$.
\end{thm}

\begin{proof}
	Let $r\in\R$ be such that
	\begin{equation}
		\label{eq:r-less-f}
		r < f(x_0).
	\end{equation}
	Let $y\in\dom f$ be arbitrary.
	
	Since $f$ is lower semicontinuous and the segment $[x_0,y]$ is compact, there is $\delta >0$ such that $f$ is bounded below on
	$$
		C := [x_0,y] + \delta B_X^\circ.
	$$
	Let $k$ be the barrier function defined by \eqref{def:k} for the set  $U:=C-\{x_0\}$.
	
	Let $a>0$ be arbitrary.
	
	Obviously,
	$$
		g(x) := -\langle x_0^*,x-x_0\rangle + a k(x-x_0)
	$$
	is a convex and continuous barrier for $U$. So, we can apply Lemma~\ref{lem:p2-reform} to  $f$ and $g$. Since
	$$
		\pa ^c g = -x_0^* +a \pa ^c k(\cdot-x_0),
	$$
	there are $x_n^*\in \pa f(x_n)$ and $y_n^*\in \pa ^c k(\cdot-x_0)(y_n)$ such that \eqref{eq:t-l:xy:close} and \eqref{eq:t-l:inf} are fulfilled and, moreover,
	\begin{equation}
		\label{eq:main:1}
		\langle x_n^*,x_n-x_0\rangle +\langle -x_0^* +ay_n^*,y_n-x_0\rangle\to 0.
	\end{equation}
	Clearly, $\langle x_0^*,x_n-y_n\rangle\to 0$, so \eqref{eq:main:1} is equivalent to
	$$
		\langle x_n^*-x_0^*,x_n-x_0\rangle + a\langle y_n^*,y_n-x_0\rangle \to 0.
	$$
	This and \eqref{eq:mon-rel} give
	\begin{equation}
		\label{eq:main:2}
		\limsup_{n\to\infty} \langle y_n^*,y_n-x_0\rangle \le 0.
	\end{equation}
    Since $y_n^*\in \pa^c k(\cdot -x_0)(y_n)$, we have
    \begin{equation}
		\label{eq:main:3}
    	k(0) \ge k(y_n-x_0) + \langle y_n^*, x_0-y_n\rangle \iff \langle y_n^*, y_n-x_0\rangle  \ge  k(y_n-x_0),
   \end{equation}
    because $k(0)=0$. But $k(y_n-x_0)\ge b\|y_n-x_0\|$ (cf.  Lemma~\ref{lem:k_proberties} (a)), so \eqref{eq:main:2} and \eqref{eq:main:3} imply that
    $$
    		\limsup_{n\to\infty} \| y_n-x_0\| \le 0\iff \lim _{n\to\infty} y_n = x_0.
    $$
    From \eqref{eq:t-l:xy:close} it follows that $x_n\to x_0$ as well; and from the lower semicontinuity of $f$ and \eqref{eq:t-l:inf} it follows that
    $$
    	f(x_0)+g(x_0) = \inf _U (f+g).
    $$
    In particular,
    $$
    	f(y) - \langle x_0^*,y-x_0\rangle +a k(y-x_0)\ge f(x_0).
    $$
    Since $a>0$ was arbitrary,
    $$
    	f(y) \ge f(x_0)+  \langle x_0^*,y-x_0\rangle.
    $$
    So, $x_0\in \dom f$ and, since $y\in\dom f$ was arbitrary, $x_0^*\in\pa ^c f(x_0)$.
\end{proof}

After all the above development, the proof of Correa-Jofr\'e-Thibault is now almost immediate.

\setcounter{thm}{1}

\begin{thm}[Correa-Jofr\'e-Thibault]
Let $X$ be a Banach space and $\pa $ be a feasible subdifferential.

Let $f:X\to\R\cup\left\{ +\infty\right\} $ be a proper lower semicontinuous function.

If $\pa f$ is monotone, then $f$ is convex.
\end{thm}

\begin{proof}
	Consider $g:X\to \Ri$ defined as
	$$
		g(\ol x) := \sup\{f(x) + \langle x^*,\ol x-x\rangle:\ (x,x^*)\in\pa f \}.
	$$
	As a supremum of linear functions, $g$ is convex and lower semicontinuous.
	
	By \eqref{eq:conv:form} we have that
	$$
		f\le g.
	$$
	From Theorem~\ref{thm:jl:mon} and monotonicity of $\pa f$ we have that $\pa f\subset \pa ^c f$, which implies $f\ge g$. Therefore,
	$$
		f = g.
	$$
\end{proof}

\end{document}